\documentclass{birkjour}
\usepackage{graphicx}
\usepackage{graphics}
\usepackage{amssymb, amsmath, amsfonts}
\usepackage{comment}

 \newtheorem{theorem}{Theorem}[section]
 \newtheorem{corollary}[theorem]{Corollary}
 \newtheorem{lemma}[theorem]{Lemma}
 \theoremstyle{definition}
 \newtheorem{definition}[theorem]{Definition}
 \theoremstyle{remark}
 \newtheorem{remark}[theorem]{Remark}
 \newtheorem*{example}{Example}
 \numberwithin{equation}{section}

\newcommand{\toto}{\rightrightarrows}
\newcommand{\R}{{\mathbb R}}
\newcommand{\N}{{\mathbb N}}

\newcommand{\ds}{\displaystyle}
\newcommand{\To}{\longrightarrow}

\def\1{\^{\i}}
\def\2{\u{a}}
\def\3{\c{s}}
\def\4{\^{a}}
\def\5{\c{t}}
\def\a{\alpha}

\def\e{\epsilon}

\def\l{\lambda}
\def\<{\langle}
\def\>{\rangle}

\DeclareMathOperator*\cl{cl}
\DeclareMathOperator*\co{co}

\DeclareMathOperator*\inte{int}

\begin{document}
\title{A Primal-Dual Approach of Weak Vector  Equilibrium Problems}



\author{Szil\'ard L\' aszl\' o}

\address{Department of Mathematics\\ Technical University of Cluj-Napoca\\
              Str. Memorandumului nr. 28, 400114 Cluj-Napoca, Romania.}
              \email{laszlosziszi@yahoo.com}



\begin{abstract}
In this paper we provide some new sufficient conditions that ensure the existence of the solution of a weak vector equilibrium problem in Hausdorff topological vector spaces ordered by a cone.  Further, we introduce a dual problem and we provide conditions that assure the solution set of the original problem and its dual coincide. We show that many known problems from the literature can be treated in our primal-dual model. We provide several coercivity conditions in order to obtain solution existence of the primal-dual problems without compactness assumption. We pay a special attention to the case when the base space is a reflexive Banach space.  We apply the results obtained to perturbed vector equilibrium problems.
\end{abstract}

\maketitle

\section{Introduction}

Equilibrium problems  provide a unified framework for treating optimization problems, fixed points, saddle points or variational inequalities as well as many important problems in physics and mathematical economics, such as location problems or Nash equilibria in game theory. The foundation of scalar equilibrium theory has been laid down by Ky Fan \cite{Fan1}, his minimax inequality still being considered one of the most notable results in this field.
 The classical scalar equilibrium problem  \cite{BO,Fan1}, described by a bifunction %
$\varphi :K\times K \longrightarrow {\mathbb{R}}$, consists in finding $x_0\in K$ such
that
\begin{equation*}
\varphi(x_0,y)\ge 0,\,\forall y\in K.
\end{equation*}

Starting with the pioneering work of Giannessi \cite{G1}, several extensions of the scalar equilibrium problem  to the vector case have been considered. These vector equilibrium problems, much like their scalar counterpart, offer a unified framework for treating vector optimization, vector variational inequalities or cone saddle point problems, to name just a few \cite{A,AKY,AKY1,AOS,Go1,G,GRTZ,La}.

Let $X$ and  $Z$ be  Hausdorff topological vector spaces, let $K\subseteq X$ be a nonempty set and let $C\subseteq Z$ be a  convex and pointed cone. Assume that the interior of the cone $C$, denoted by $\inte C$, is nonempty and consider the mapping $F:K\times K\times K\To Z.$  The weak vector equilibrium problem governed by the vector trifunction $F$  consists in finding $x_0\in K$, such that
\begin{equation}\label{p1}
F(x_0,y,x_0)\not\in-\inte C,\,\forall y\in K.
\end{equation}

The dual vector equilibrium problem of (\ref{p1}) is defined as:
 Find $x_0\in K$, such that
\begin{equation}\label{p2}
F(x_0,y,y)\not\in-\inte C,\,\forall y\in K.
\end{equation}

It can easily be observed, that for $Z=\R$ and $C=\R_+=[0,\infty),$ the previous problems reduce to the scalar equilibrium  problems studied by Inoan and  Kolumb\'an in \cite{KI}.

The study of the problems (\ref{p1}) and (\ref{p2}) is motivated by the following setting.
Assume that the weak vector equilibrium problem, which consists in finding $x_0\in K$ such that $f(x_0,y)\not\in-\inte C,$ has no solution though the diagonal condition $f(x,x)=0,\,\forall x\in K$ holds. Then, we may study instead a  perturbed equilibrium problem (see also \cite{durea,FQ}) and provide assumptions on the perturbation function $g$, such that the problem which consists in finding $x_0\in K$ such that $f(x_0,y)+g(x_0,y)\not\in-\inte C,$  has a solution.
But in this case the latter problem is the dual of the following problem: Find $x_0\in K$, such that $F(x_0,y,x_0)\not\in-\inte C,$ for all $y\in K,$ with the trifunction $F(x,y,z)=f(x,z)+g(x,y).$ Moreover, for an appropriate perturbation $g$ the primal problem, that is, find $x_0\in K$ such that $g(x_0,y)\not\in-\inte C$ has a solution. Hence, it is  worthwhile to obtain conditions that assure the the solution sets of (\ref{p1}) and (\ref{p2}) coincide. This setting may have some important consequences. Indeed, by taking $X$ a Banach space and $g(x,y)=\e\|x-y\|e$, where $\e>0$ and $e\in C\setminus\{0\},$ a solution of the perturbed vector equilibrium problem is called $\e$-equilibrium point, see \cite{BKP,BKP1}. Further, special cases of the perturbed vector equilibrium problems lead to some deep results such as Deville-Godefroy-Zizler perturbed equilibrium principle or Ekeland vector variational principle, see \cite{FQ}.

Moreover, take $F(x,y,z)=\<Az,y-x\>,$ where $A:K\To L(X,Z)$ is a given operator and $L(X,Z)$ denotes the set of all linear and continuous operators
from $X$ to $Z$.  For $x^*\in L(X,Z)$ and $x\in X$, we denote by $\<x^*,x\>$ the vector $x^*(x)\in Z.$ In this setting (\ref{p1}) becomes:  find $x_0\in K$, such that $\<Ax_0,y-x_0\>\not\in-\inte C$ for all $y\in K,$ which is the weak vector variational inequality of Stampacchia, see \cite{Fe}.
On the other hand (\ref{p2}) becomes:  find $x_0\in K$, such that $\<Ay,y-x_0\>\not\in-\inte C$ for all $y\in K,$ which is the weak vector variational inequality of Minty, see \cite{Fe}.

In this paper, we obtain some existence results of the solution for the vector equilibrium problem (\ref{p1}) and (\ref{p2}). Some of our conditions are new in the literature. Several examples and counterexamples circumscribe our research and show that our conditions are essential.
The paper is organized as follows. In the next section, we introduce some preliminary notions and the necessary apparatus that we need in order to obtain our results. In section 3 and section 4 we state our results concerning on weak vector equilibrium problems. Our conditions, which ensure the solution existence of the above mentioned vector equilibrium problems are  considerably weakening the existing conditions from the literature.  We pay a special attention to the case when the set $K$ is a closed subset of a reflexive Banach  space.  Finally, we apply our results to vector equilibrium problems given by the sum of two bifunctions which  can be seen as perturbed equilibrium problems.

\section{Preliminaries}

Let $X$ be a real Hausdorff topological vector space.  For a non-empty set $D\subseteq X$, we denote by $\inte D$ its interior, by $\cl D$ its closure and by $\co D$ its convex hull.   Recall that a set $C\subseteq X$ is a cone, iff $tk\in C$ for all $c\in C$   and $t\ge 0.$ The cone $C$ is convex if $C+C=C,$ and pointed if $C\cap (-C)=\{0\}.$ Note that a closed, convex and pointed cone $C$ induce a partial ordering on $Z$, that is $z_1\le z_2\Leftrightarrow z_2-z_1\in C.$ In the sequel when we use $\inte C,$  we tacitly assume that the cone $C$ has nonempty interior. Following the same logical approach, one can introduce the strict inequality  $z_1< z_2\Leftrightarrow z_2-z_1\in \inte C,$ or $z_1< z_2\Leftrightarrow z_2-z_1\in  C\setminus\{0\}.$ These relations lead to  saying, that $z_1\not< z_2\Leftrightarrow z_2-z_1\not\in-\inte C$, or $z_1\not< z_2\Leftrightarrow z_2-z_1\not\in-C\setminus\{0\}.$ It is an easy exercise to show that   $\inte C+C= \inte C.$

Let   $Z$ be  another  Hausdorff topological vector space, let $K\subseteq X$ be a nonempty set and let $C\subseteq Z$ be a  convex and pointed cone.

A map $f:K\subseteq X\To Z$ is said to be C-upper semicontinuous at $x\in K$ iff for any neighborhood $V$ of $f(x)$ there exists a neighborhood $U$ of $x$ such that $f(u)\in  V-C$ for all $u\in U\cap K$. Obviously, if $f$ is continuous at $x\in K,$ then it is  also C-upper semicontinuous at $x\in K$.
Assume that $C$ has nonempty interior.  According to \cite{Ta} $f$ is C-upper semicontinuous at $x\in K,$ if and only if, for any $k\in\inte C$, there exists a neighborhood $U$ of $x$ such that
$$f(u) \in f(x) + k -\inte C\mbox{ for all }u \in U\cap K.$$

The map $f:K\To Z$ is said to be C-lower semicontinuous at $x\in K$ iff the map $-f$ is C-upper semicontinuous at $x\in K.$

\begin{definition} Let $K\subseteq X$ be convex. The function $f : K\to Z$ is called $C$-convex on $K$, iff  for all $x,y\in K$ and $t\in[0,1]$ one has
$$tf(x)+(1-t)f(y)-f(tx+(1-t)y)\in C.$$
\end{definition}

Note that the function $f : K\to Z$ is  $C$-convex, iff  for all $x_1,x_2,\ldots, x_n\in K$, $n\in \N$ and $\lambda_i\ge 0,\, i\in\{1,2,\ldots,n\},$ with
$\sum_{i=1}^n \lambda_i=1,$  one
has
$$ \sum_{i=1}^n \l_i f(x_i)-f\left(\sum_{i=1}^n \l_ix_i\right) \in C.$$
We will  use the following notations for the open, respectively
closed, line segments in $X$ with the endpoints $x$ and $y$
\begin{eqnarray*}
]x,y[ &:=&\big\{z\in X:z=x+t(y-x),\,t\in ]0,1[\big\}, \\
\lbrack x,y] &:=&\big\{z\in X:z=x+t(y-x),\,t\in \lbrack 0,1]\big\}.
\end{eqnarray*}
The line segments $]x,y],$ respectively $[x,y[$ are defined similarly. Further, we need the following notions see \cite{G1}.
 \begin{definition}  Let $X$ and $Z$ be  Hausdorff  topological vector spaces,  let $C\subseteq Z$ be a convex and pointed cone with nonempty interior and let $K$ be a nonempty subset of $X$. Consider the mapping $F :K \times K\times K \longrightarrow{Z}$. We say that $F$ is weakly C-pseudomonotone with respect to the third variable, if for all $x,y\in K$
  $$F(x,y,x)\not\in-\inte C\implies F(x,y,y)\not\in-\inte C.$$
\end{definition}

\begin{definition} Let $X$ and $Z$ be  Hausdorff  topological vector spaces,  let $C\subseteq Z$ be a convex and pointed cone with nonempty interior and let $K$ be a nonempty, convex subset of $X$. Consider the mapping $F :K \times K\times K \longrightarrow{Z}$. We say that $F$ is weakly explicitly C-quasiconvex with respect to the second variable, if for all $x,y,z\in K$ and for all $t\in]0,1[$ one has
$$F(x,(1-t)x+ty,z)-F(x,x,z)\in-\inte C,$$
or
$$F(x,(1-t)x+ty,z)-F(x,y,z)\in-\inte C.$$
\end{definition}

\begin{definition}  Let $X$ and $Z$ be  Hausdorff  topological vector spaces,  let $C\subseteq Z$ be a convex and pointed cone with nonempty interior and let $K$ be a nonempty, convex subset of $X$. Consider the mapping $F :K \times K\times K \longrightarrow{Z}$. We say that $F$ is weakly C-hemicontinuous with respect to the third variable, if for  all $x,y\in K$ such that  $F(x,y,(1-t)x+ty)\not\in-\inte C$ for all $t\in ]0,1]$ one has
 $$F(x,y,x)\not\in-\inte C.$$
\end{definition}
In subsequent section, the notion of a KKM map and the well-known intersection Lemma due to Ky Fan %
\cite{Fan} will be needed.

\begin{definition}
(Knaster-Kuratowski-Mazurkiewicz) Let $X$ be a Hausdorff topological vector
space and let $M\subseteq X.$ The application $G:M\rightrightarrows X$ is
called a KKM application if for every finite number of elements $%
x_1,x_2,\dots,x_n\in M$ one has $$\co\{x_1,x_2,\ldots,x_n\}\subseteq \bigcup_{i=1}^n G(x_i).$$
\end{definition}

\begin{lemma}[Fan \cite{Fan}]\label{fan}
Let $X$ be a Hausdorff topological vector space, $M\subseteq X$ and $%
G:M\rightrightarrows X$ be a KKM application. If $G(x)$ is closed for every $%
x\in M$, and there exists $x_{0}\in M,$ such that $G(x_{0})$ is compact,
then $\bigcap_{x\in M}G(x)\neq \emptyset.$
\end{lemma}

\section{The coincidence of solution sets and solution existence}
In this section we provide several conditions, some of them new in the literature, that assure the existence of solution of problem (\ref{p1}) and (\ref{p2}), respectively. Further, we give conditions that assure the coincidence of the solution sets of these problems. Hence, we can deduce the solution existence of the dual problem from the nonemptyness of the solution set of the primal problem and viceversa.

In what follows we provide a Minty type result (see \cite{Fe,Mau-Rac}) for the problems (\ref{p1}) and (\ref{p2}). More precisely, we provide conditions that assure the coincidence of the solutions set of problem (\ref{p1}) and (\ref{p2}), respectively.


\begin{theorem}\label{t3.0}   Let $X$ and $Z$ be  Hausdorff  topological vector spaces,  let $C\subseteq Z$ be a convex and pointed cone with nonempty interior and let $K$ be a nonempty subset of $X$. Consider the mapping $F :K \times K\times K \longrightarrow{Z}$. Then, the following statements hold.
\begin{itemize}
\item[(i)] If $F$ is weakly C-pseudomonotone with respect to the third variable, then every $x\in K$ which solves $(\ref{p1})$ is also a solution of $(\ref{p2})$.
 \item[(ii)] Assume that $K$ is convex and $F(x,x,y)\in -C$ for all $x,y\in K,\,x\neq y$. If $F$  is weakly explicitly C-quasiconvex with respect to the second variable and is weakly C-hemicontinuous with respect to the third variable, then every $x\in K$ which solves $(\ref{p2})$ is also a solution of $(\ref{p1}).$
     \end{itemize}
\end{theorem}
\begin{proof} $"(i)"$ Let $x_0\in K$ be a solution of (\ref{p1}). Then $F(x_0,y,x_0)\not\in-\inte C$ for all $y\in K.$ On the other hand $F$ is  weakly C-pseudomonotone with respect to the third variable, hence $F(x_0,y,x_0)\not\in-\inte C$ implies $F(x_0,y,y)\not\in-\inte C$ for all $y\in K$.

$"(ii)"$ Let $x_0\in K$ be a solution of (\ref{p2}). Then $F(x_0,y,y)\not\in-\inte C$ for all $y\in K.$ Let $z\in K,\,z\neq x_0.$ Since $K$ is convex we obtain that $(1-t)x_0+tz\in K$ for all $t\in [0,1].$

Consequently, we have $$F(x_0,(1-t)x_0+tz,(1-t)x_0+tz)\not\in-\inte C$$ for all $t\in [0,1].$

But $F$ is weakly explicitly quasiconvex relative the second variable, hence for all $t\in ]0,1[$ one has
$$F(x_0,(1-t)x_0+tz,(1-t)x_0+tz)-F(x_0,x_0,(1-t)x_0+tz)\in -\inte, C$$
or
$$F(x_0,(1-t)x_0+tz,(1-t)x_0+tz)-F(x_0,z,(1-t)x_0+tz)\in -\inte C.$$

Since $F(x,x,y)\in -C$ for all $x,y\in K,\,x\neq y$ and $C+\inte C=\inte C$ the first relation cannot hold. Hence, for all $t\in ]0,1[$ one has

$$F(x_0,(1-t)x_0+tz,(1-t)x_0+tz)-F(x_0,z,(1-t)x_0+tz)\in -\inte C.$$
Since  $F(x_0,(1-t)x_0+tz,(1-t)x_0+tz)\not\in-\inte C$, for all $t\in [0,1],$ and by assumption $F(x_0,z,z)\not\in-\inte C$, we have that  for all $t\in ]0,1],$ $$F(x_0,z,(1-t)x_0+tz)\not\in -\inte C.$$

Taking into account the fact that $F$ is weakly C-hemicontinuous with respect to the third variable,  we obtain $$F(x_0,z,x_0)\not\in-\inte C.$$  Since $z$ is arbitrary, it follows that $x_0$ is a solution of (\ref{p1}).
\end{proof}

An immediate consequence is the following.
\begin{corollary}  Let $X$ and $Z$ be  Hausdorff  topological vector spaces,  let $C\subseteq Z$ be a convex and pointed cone with nonempty interior and let $K$ be a convex nonempty subset of $X$. Consider the mapping $F :K \times K\times K \longrightarrow{Z}$. Assume that $F(x,x,y)\in -C$ for all $x,y\in K,\,x\neq y$. Assume further that $F$  is weakly explicitly C-quasiconvex with respect to the second variable, $F$ is weakly C-pseudomonotone with respect to the third variable and $F$ is weakly C-hemicontinuous with respect to the third variable. Then the solution sets of the problems $(\ref{p1})$ and $(\ref{p2})$ coincide.
\end{corollary}
\begin{remark}\rm
Note that the assumptions $F(x,x,y)\in -C$ for all $x,y\in K,\,x\neq y$ and $F$  is weakly explicitly C-quasiconvex with respect to the second variable in the hypothesis (ii) of Theorem \ref{t3.0} can be replaced by the assumption
$$F(x,(1-t)x+ty,z)-F(x,y,z)\in-\inte C,$$
for all $t\in]0,1[,$ as follows directly from the proof.

However, in what follows we show that the latter assumption 
is essential. More precisely, we give an example of a trifunction $F$  which is not weakly explicitly C-quasiconvex with respect to the second variable but all the other assumptions of Theorem \ref{t3.0} hold, meanwhile the problem (\ref{p2}) has a solution, but the problem (\ref{p1}) has no  solution.
\end{remark}
\begin{example}\rm [see also \cite{L1}, Example 3.2] \label{ex3.1} Let us consider the trifunction $F:[-1,1]\times [-1,1]\times [-1,1]\To \R^2,$
 $$F(x,y,z)=\left((f(y)-f(x))g(z),(f(y)-f(x))g(z)\right),$$
 where  $f:[-1,1]\To[0,1],$  $f(x)=\left\{
\begin{array}{lll}
-2x-1,\,\mbox{if\,}\, x\in\left[-1,-\ds\frac12\right]\\
2x+1,\,\mbox{if\,}\, x\in\left(-\ds\frac12,0\right]\\
-2x+1,\,\mbox{if\,}\,x\in\left(0,1\right],\\
\end{array}
\right.$
and

$g:[-1,1]\To [-1,1],$
$ g(x)=\left\{
\begin{array}{ll}
-\ds\frac23 x+\frac13,\,\mbox{if,}\, x\in\left[-1,\ds\frac12\right]\\
-2x+1,\,\mbox{if,}\, x\in\left(\ds\frac12,1\right].\\
\end{array}
\right.$
\end{example}
Further, consider $C=\R_+^2=\{(x_1,x_2)\in\R^2:x_1\ge 0,\,x_2\ge0\}$ the nonnegative orthant of $\R^2$, which is obviously a convex and pointed cone, with nonempty interior. We  consider the problem (\ref{p1}) and (\ref{p2}) defined by the trifunction $F$ and by the cone $C.$
Obviously the set $K=[-1,1]$ is convex. Further $F(x,x,y)=(0,0)\in -C$ for all $x,y\in K$ and since the functions $f$ and $g$ are continuous, from  $F(x,y,(1-t)x+ty)\not\in-\inte C$ for all $t\in ]0,1]$ one has
 $F(x,y,x)\not\in-\inte C,$ by taking the limit $t\To 0.$ Hence, $F$ is weakly C-hemicontinuous with respect to the third variable.
 We show that $F$ is not weakly explicitly C-quasiconvex with respect to the second variable. Indeed, for $x=-1,\,y=-\frac12,\,z=\frac12$ and all $t\in]0,1[$ one has $F(x,(1-t)x+ty,z)-F(x,x,z)=(0,0)\not\in-\inte C$ and
$F(x,(1-t)x+ty,z)-F(x,y,z)=(0,0)\not\in-\inte C.$

 We show that  $x_0=-\frac12\in K$ is a solution of the (\ref{p2}), but is not a solution of (\ref{p1}).
Indeed, it can easily  be verified that $\left(f(y)-f\left(-\frac12\right)\right)\cdot g(y)\ge 0$ for all $y\in K,$ hence $F\left(-\frac12,y,y\right)\not\in-\inte C.$ In other words $x_0=-\frac12$ is a solution of (\ref{p2}).

 On the other hand, for $y=\frac34\in K$ we obtain $\left(f(y)-f\left(-\frac12\right)\right)\cdot g\left(-\frac12\right)=-\frac12\cdot\frac23<0$ which shows that $F\left(-\frac12,y,-\frac12\right)\in-\inte C.$  Hence, $x_0=-\frac12$ is not a solution of (\ref{p1}).

\begin{remark}\rm
In order to use Fan's Lemma to obtain solution existence for the problem (\ref{p1}) we need conditions that assure for every $y\in K$ the closedness of the sets $G(y)=\{x\in K: F(x,y,x)\not\in-\inte C\}.$
\end{remark}
\begin{lemma}\label{primclosed} Let $X$ and $Z$ be  Hausdorff  topological vector spaces,  let $C\subseteq Z$ be a convex and pointed cone with nonempty interior and let $K$ be a nonempty, convex and closed subset of $X$. Let $y\in K$ and consider the mapping $F :K \times K\times K \longrightarrow{Z}.$ Assume that one of the following conditions hold.
\begin{itemize}
\item[(a)] For every  net $(x_\a)\subseteq K,\, \lim x_\a=x$ one has:
$F(x_\a,y,x_\a)\not\in-\inte C\implies F(x,y,x)\not\in-\inte C.$
\item[(b)] The mapping $x\longrightarrow F(x,y,x)$ is C-upper semicontinuous on $K.$
\item[(c)] For every  $x\in K$ and for every  net $(x_\a)\subseteq K,\, \lim x_\a=x$ there exists a net $z_\a\subseteq Z,\, \lim z_\a=z$ such that $F(x_\a,y,x_\a)-z_\a\in -C$ and $F(x,y,x)-z\in C.$
\end{itemize}
Then, the set $G(y)=\{x\in K: F(x,y,x)\not\in-\inte C\}$ is closed.
\end{lemma}
\begin{proof} Note that $(a)$ is obvious. Let us prove (b). Consider the net $(x_\a)\subseteq G(y)$ and let $\lim x_\a=x_0.$ Assume that $x_0\not\in G(y).$ Then $F(x_0,y,x_0)\in-\inte C$.
 According to the assumption the function $x\longrightarrow F(x,y,x)$ is C-upper semicontinuous at $x_0$, hence  for every $k\in\inte C$ there exists $U,$ a neighborhood of $x_0,$ such that $F(x,y,x)\in F(x_0,y,x_0)+k-\inte C$ for all $x\in U.$  But then, for $k=-F(x_0,y,x_0)\in\inte C$,  one obtains that there exits $\a_0$ such that $F(x_\a,y,x_\a)\in -\inte C,$ for $\a\ge\a_0,$ which contradicts the fact that $(x_\a)\subseteq G(y_0)$.
Hence $G(y)\subseteq K$ is closed.

For (c) consider the net $(x_\a)\subseteq G(y)$ and let $\lim x_\a=x_0.$ Assume that $x_0\not\in G(y).$ Then $F(x_0,y,x_0)\in-\inte C$. But by the assumption there exists a net $z_\a\subseteq K,\, \lim z_\a=z$ such that $F(x_\a,y,x_\a)-z_\a\in -C$ and $F(x_0,y,x_0)-z\in C.$ From the latter relation we get $z\in-\inte C$, and since $-\inte C$ is open we have $z_\a\in -\inte C$ for every $\a\ge \a_0.$ But then, $F(x_\a,y,x_\a)\in z_\a -C$ and $\inte C+C=\inte C$ leads to $F(x_\a,y,x_\a)\in -\inte C$ for $\a\ge\a_0,$ contradiction.
\end{proof}

\begin{remark}\rm  Note that condition (c) seems to be new in the literature. In what follows we show that (b) implies (c).

Assume that for a fixed $y\in K$, the mapping $x\longrightarrow F(x,y,x)$ is C-upper semicontinuous on $K.$
    Let $x_0\in K$ and consider the  net $(x_\a)\subseteq K,\, \lim x_\a=x_0.$ We show that there exists a net $z_\a\subseteq Z,\, \lim z_\a=z$ such that $F(x_\a,y,x_\a)-z_\a\in -C$ and $F(x_0,y,x_0)-z\in C.$

    We have that for every neighbourhood of $F(x_0,y,x_0)$, say $V$, there exists $U,$ a neighbourhood of $x_0$, such that $F(x,y,x)\in V-C$ for all $x\in U.$ Obviously on can take $V$ closed, hence there exists a net $(s_\a)\subseteq V$ such that $\lim s_\a=F(x_0,y,x_0).$
 Since $\lim x_\a=x_0$  we have that $x_\a\in U$ for every $\a\ge\a_0,$ hence $F(x_\a,y,x_\a)\in V-C$ for every $\a\ge\a_0.$ This leads to
 $F(x_\a,y,x_\a)-s_\a\in -C$ for every $\a\ge\a_0.$ Hence one can take $z_\a=F(x_\a,y,x_\a)$ if $\a<\a_0$ and $z_\a=s_\a$ for $\a\ge\a_0$ and the conclusion follows.
\end{remark}

In what follows we provide an example to emphasize that the condition (c) in Lemma \ref{primclosed} is in general weaker than condition (b).
\begin{example}\rm\label{ex3.2} Let $C=\{(x_1,x_2)\in \R^2: x_1^2\le x_2^2,\,x_2\ge 0\}.$ Obviously $C$ is a closed convex and pointed cone in $\R^2$ with nonempty interior. Consider the trifunction $$F:[0,1]\times [0,1]\times [0,1]\To \R^2,F(x,y,z)=\left\{\begin{array}{ll}(x+y,2x-z)\mbox{ if } 0\le x\le\frac12,\\(2x+y,z-x)\mbox{ if }\frac12<x\le 1.\end{array} \right.$$
Then, for every fixed $y\in [0,1]$ the mapping $x\To F(x,y,x)$ is continuous, hence it is also C-upper semicontinuous, at every $x\in[0,1]\setminus\{\frac12\}.$
We show that $x\To F(x,y,x)$ is not C-upper semicontinuous at the point $x=\frac12$ for every fixed $y\in[0,1].$ For this it is enough to show that for all $\e>0$ there exists $(c_1,c_2)\in\inte C$ and $x\in \left]\frac12-\e,\frac12+\e\right[$ such that $(c_1,c_2)+F(\frac12,y,\frac12)-F(x,y,x)\not\in\inte C.$ Hence, for fixed $\e>0$ ($\e<2$) let $(c_1,c_2)=(\e-1,1)\in\inte C.$ Consider $x=\frac12+\frac{\e}{2}\in \left]\frac12-\e,\frac12+\e\right[.$ Then, $(c_1,c_2)+F(\frac12,y,\frac12)-F(x,y,x)=\left(-\frac32,\frac32\right)\not\in\inte C.$

Next, we show that condition (c) in Lemma \ref{primclosed} holds for $x=\frac12$ and every fixed $y\in[0,1].$ Obviously instead of nets one can consider sequences, hence let $(x_n)\subseteq [0,1],\,x_n\To \frac12,\,n\To\infty.$ We must show, that there exists a sequence $(z_n)\subseteq \R^2,\, \lim z_n=z$ such that $F(x_n,y,x_n)-z_n\in -C$ and $F\left(\frac12,y,\frac12\right)-z\in C.$

Let $z_n=(x_n+y,x_n).$ Then $F(x_n,y,x_n)-z_n=(0,0)\in-C$ for every $n\in\N,$ such that $x_n\le\frac12$ and $F(x_n,y,x_n)-z_n=(x_n,-x_n)\in-C$ for every $n\in\N,$ such that $x_n>\frac12.$ Obviously $\lim z_n=z=\left(\frac12+y,\frac12\right)$, hence $F\left(\frac12,y,\frac12\right)-z=(0,0)\in C.$
\end{example}

Now we are able to prove the following existence result concerning on the solution of the problem (\ref{p1}).

\begin{theorem}
\label{t11} Let $X$ and $Z$ be  Hausdorff topological vector spaces,  let $C\subseteq Z$ be a convex and pointed cone with nonempty interior, and let $K$ be a nonempty, convex and closed subset of $X$. Consider the mapping $F :K \times K\times K \longrightarrow{Z}$  satisfying
\begin{itemize}
\item[(i)] $\forall y\in K,$ one of the conditions (a), (b), (c) in Lemma \ref{primclosed} is satisfied,

\item[(ii)] $\forall x\in K,$ the mapping $y\longrightarrow F(x,y,x)$ is C-convex,

\item[(iii)] $\forall x\in K,\, F(x,x,x)\not\in-\inte  C,$

\item[(iv)] There exists $K_0\subseteq X$ a nonempty and compact set and $y_0\in K,$ such that $F(x,y_0,x)\in-\inte C,$ for all $x\in K\setminus K_0.$
\end{itemize}

Then, there exists an element $x_0\in K$ such that $F(x_0,y,x_0)\not\in -\inte{C},$ for all $y\in K.$
\end{theorem}
\begin{proof}  We consider the set-valued map $G:K\toto K,\,\, G(y)=\{x\in K:F(x,y,x)\not\in-\inte C\}.$ From (i) via Lemma \ref{primclosed} we obtain that $G(y)$ is closed for all $y\in K.$ Moreover, (iii) assures that $G(y)\neq\emptyset,$ since $y\in G(y).$

We show next that $G$ is a  KKM mapping.
Assume the contrary. Then, there exists $y_1,y_2,...,y_n\in K$ and $y\in \co\{y_1,y_2,...,y_n\}$ such that $y\not\in\cup_{i=1}^n G(y_i).$ In other words,
there exists $\l_1,\l_2,...,\l_n\ge 0$ with $\sum_{i=1}^n \l_i=1$ such that $y=\sum_{i=1}^n \l_i y_i\not\in G(y_i)$ for all $i\in\{1,2,...,n\},$ that is
$F\left(\sum_{i=1}^n \l_i y_i,y_i,\sum_{i=1}^n \l_i y_i\right)\in-\inte C,$ for all $i\in\{1,2,...,n\}.$
But then, since $-\inte C$ is convex, one has
$\sum_{i=1}^n \l_i F\left(\sum_{i=1}^n \l_i y_i,y_i,\sum_{i=1}^n \l_i y_i\right)\in-\inte C.$

From assumption (ii),  we have that $$\sum_{i=1}^n \l_i F\left(\sum_{i=1}^n\l_i y_i,y_i,\sum_{i=1}^n\l_i y_i\right)- F\left(\sum_{i=1}^n\l_i y_i, \sum_{i=1}^n\l_i y_i,\sum_{i=1}^n\l_i y_i\right )\in C,$$ or equivalently,  $$F\left(\sum_{i=1}^n\l_i y_i, \sum_{i=1}^n\l_i y_i,\sum_{i=1}^n\l_i y_i\right )\in\sum_{i=1}^n \l_i F\left(\sum_{i=1}^n\l_i y_i,y_i,\sum_{i=1}^n\l_i y_i\right)- C.$$

 On the other hand, $\sum_{i=1}^n \l_i F\left(\sum_{i=1}^n\l_i y_i,y_i,\sum_{i=1}^n\l_i y_i\right)\in-\inte C$ and $\inte C+C=\inte C$, hence
$$F\left(\sum_{i=1}^n\l_i y_i, \sum_{i=1}^n\l_i y_i,\sum_{i=1}^n\l_i y_i\right )\in -\inte C,$$  which contradicts (iii). Consequently,  $G$ is a KKM application.

 We show that $G(y_0)$ is compact. For this is enough to show that $G(y_0)\subseteq K_0.$ Assume the contrary, that is $G(y_0)\not\subseteq K_0.$ Then, there exits $z\in G(y_0)\setminus K_0.$ This implies that $z\in K\setminus K_0,$ and according to (iv) $F(z,y_0,z)\in-\inte C,$ which contradicts the fact that $z\in G(y_0).$

Hence, $G(y_0)$ is a closed subset of the compact set $K_0$ which shows that $G(y_0)$ is compact.

Thus, according to Ky Fan's Lemma, $\bigcap_{y\in  K}G(y)\neq\emptyset.$ In other words, there exists $x_0\in K,$ such that $F(x_0,y,x_0)\not\in-\inte  C$ for all $y\in K.$
\end{proof}

\begin{remark}\rm   The approach, based on Ky Fan's Lemma,  in the proof of Theorem \ref{t11}, is well known in the literature, see, for instance,
\cite{ACY,AK,L1,La1,La,DTL}. Note that condition (iv) combined with condition (iii) in Theorem \ref{t11} ensure that $y_0\in K\cap K_0$, hence $K\cap K_0\neq \emptyset,$  and since  $K\cap K_0$ is compact one can assume directly that $K_0\subseteq K.$ Further, if $K$ is compact condition (iv) is automatically satisfied with $K_0=K.$
\end{remark}

In what follows, inspired from \cite{kazmi},  we provide another coercivity condition concerning a compact set and its algebraic interior.
Let $U,V\subseteq X$ be convex sets and assume that $U\subseteq V$. We recall that the algebraic interior of $U$ relative to
$V$ is defined as
$$\mbox{core}_V U=\{u\in U: U\cap ]u,v]\neq\emptyset,\,\forall v\in V \}$$ Note that $\mbox{core}_V V = V.$
Our coercivity condition concerning the problem (\ref{p1}) becomes:

  There exists a nonempty compact convex subset $K_0$ of $K$ such that for every $x\in K_0\setminus\mbox{core}_K K_0$ there exists an $y_0\in \mbox{core}_K K_0$ such that $F(x,y_0,x)\in -C.$

In the following results we use the coercivity conditions emphasized above and we  drop the  closedness condition  on $K$. However, condition (iii) also changes.

 \begin{theorem}
\label{t110} Let $X$ and $Z$ be  Hausdorff topological vector spaces,  let $C\subseteq Z$ be a convex and pointed cone with nonempty interior, and let $K$ be a nonempty, convex subset of $X$. Consider the mapping $F :K \times K\times K \longrightarrow{Z}$  satisfying

\begin{itemize}
\item[(i)]$\forall y\in K,$ one of the conditions (a), (b), (c) in Lemma \ref{primclosed} is satisfied,

\item[(ii)] $\forall x\in K,$ the mapping $y\longrightarrow F(x,y,x)$ is C-convex,

\item[(iii)] $\forall x\in K,\, F(x,x,x)\in -C\setminus -\inte C,$

\item[(iv)] There exists a nonempty compact convex subset $K_0$ of $K$ with the property that for every $x\in K_0\setminus\mbox{core}_K K_0,$ there exists an $y_0\in \mbox{core}_K K_0$ such that $F(x,y_0,x)\in -C.$
\end{itemize}

Then, there exists an element $x_0\in K$ such that $F(x_0,y,x_0)\not\in -\inte{C},$ for all $y\in K.$
\end{theorem}
\begin{proof} $K_0$ is compact, hence, according to Theorem \ref{t11} there exists $x_0\in K_0$ such that $F(x_0,y,x_0)\not\in-\inte C,\,\forall y\in K_0.$ We show, that $F(x_0,y,x_0)\not\in-\inte C,\,\forall y\in K.$
First we show, that there exists $z_0\in \mbox{core}_K K_0$ such that $F(x_0,z_0,x_0)\in -C.$ Indeed, if $x_0\in \mbox{core}_K K_0$ then let $z_0=x_0$ and the conclusion follows from (iii). Assume now, that $x_0\in K_0\setminus\mbox{core}_K K_0. $ Then, according to (iv),  there exists $z_0\in \mbox{core}_K K_0$ such that $F(x_0,z_0,x_0)\in -C.$

Let $y\in K.$ Then, since $z_0\in \mbox{core}_K K_0$, there exists $\l\in[0,1]$ such that $\l z_0 +(1-\l)y\in K_0,$ consequently $F(x_0,\l z_0 +(1-\l)y,x_0)\not\in-\inte C.$ From (ii) we have
$$\l F(x_0,z_0,x_0)+(1-\l)F(x_0,y,x_0)-F(x_0,\l z_0 +(1-\l)y,x_0)\in C$$
or, equivalently
$$(1-\l)F(x_0,y,x_0)-F(x_0,\l z_0 +(1-\l)y,x_0)\in C-\l F(x_0,z_0,x_0)\subseteq C.$$
Assume that $F(x_0,y,x_0)\in-\inte C.$ Then, $$-F(x_0,\l z_0 +(1-\l)y,x_0)\in -(1-\l)F(x_0,y,x_0)+ C\subseteq \inte C,$$ in other words
$$F(x_0,\l z_0 +(1-\l)y,x_0)\in-\inte C,$$
contradiction. Hence,
$F(x_0,y,x_0)\not\in-\inte C,$ for all $y\in K.$
\end{proof}

\begin{remark}\rm According to Theorem \ref{t3.0}, under the extra assumption that $F$ is weakly C-pseudomonotone with respect to the third variable Theorem \ref{t11} and Theorem \ref{t110} provide the solution existence of (\ref{p2}).
\end{remark}

Using the same technique as in the proof of Theorem \ref{t11}, based on Fan's Lemma, on can easily obtain solution existence of (\ref{p2}). However, note that depending by the structure of the trifunction $F$, the conditions may significantly differ to those assumed in the hypothesis of Theorem \ref{t11} or Theorem \ref{t110}. 
In what follows we state a result  concerning the closedness of the set $G(y)=\{x\in K: F(x,y,y)\not\in-\inte C\}.$

\begin{lemma}\label{dualclosed} Let $X$ and $Z$ be  Hausdorff  topological vector spaces,  let $C\subseteq Z$ be a convex and pointed cone with nonempty interior and let $K$ be a nonempty, convex and closed subset of $X$. Let $y\in K$ and consider the mapping $F :K \times K\times K \longrightarrow{Z}.$ Assume that one of the following conditions hold.
\begin{itemize}
\item[(a)] For every  net $(x_\a)\subseteq K,\, \lim x_\a=x$ one has: $F(x_\a,y,y)\not\in-\inte C\implies F(x,y,y)\not\in-\inte C.$
\item[(b)] The mapping $x\longrightarrow F(x,y,y)$ is C-upper semicontinuous on $K.$
\item[(c)] For every $x\in K$ and for every  net $(x_\a)\subseteq K,\, \lim x_\a=x,$  there exists a net $z_\a\subseteq Z,\, \lim z_a=z,$ such that $F(x_\a,y,y)-z_\a\in -C$ and $F(x,y,y)-z\in C.$
\end{itemize}
Then, the set $G(y)=\{x\in K: F(x,y,y)\not\in-\inte C\}$ is closed.
\end{lemma}
The proof is similar to the proof of Lemma \ref{primclosed} therefore we omit it. Our coercivity condition concerning the problem (\ref{p2}) is the following:

 There exists a nonempty compact convex subset $K_0$ of $K$ such that for every $x\in K_0\setminus\mbox{core}_K K_0$ there exists an $y_0\in \mbox{core}_K K_0$ such that $F(x,y_0,y_0)\in -C.$ As we have mentioned before, it is an easy exercise to provide solution existence of (\ref{p2}) under similar conditions to those in the hypotheses of Theorem \ref{t11}  and Theorem \ref{t110}.

 However, by using Theorem \ref{t3.0} we obtain the following existence result concerning solution existence of (\ref{p1}).

\begin{theorem}\label{t112} Let $X$ and $Z$ be  Hausdorff  topological vector spaces,  let $C\subseteq Z$ be a convex and pointed cone with nonempty interior and let $K$ be a nonempty, convex subset of $X$. Consider the mapping $F :K \times K\times K \longrightarrow{Z}$  satisfying

\begin{itemize}
\item[(i)]$\forall y\in K,$ one of the conditions (a), (b), (c) in Lemma \ref{dualclosed} is satisfied,

\item[(ii)] $\forall x\in K,$ the mapping $y\longrightarrow F(x,y,y)$ is C-convex,

\item[(iii)] $\forall x\in K,\, F(x,x,x)\in -C\setminus -\inte C$ and $F(x,x,y)\in -C$ for all $x,y\in K,\,x\neq y$,

\item[(iv)] There exists a nonempty compact convex subset $K_0$ of $K$ with the property that for every $x\in K_0\setminus\mbox{core}_K K_0,$ there exists an $y_0\in \mbox{core}_K K_0$ such that $F(x,y_0,y_0)\in -C.$

 \item[(v)] $F$  is weakly explicitly C-quasiconvex with respect to the second variable,

 \item[(vi)] $F$ is weakly C-hemicontinuous with respect to the third variable.
\end{itemize}

Then, there exists an element $x_0\in K$ such that $F(x_0,y,x_0)\not\in -\inte{C}$ for all $y\in K.$
\end{theorem}
 \begin{proof} Similarly to the proof of Theorem \ref{t110} one can prove that (i)-(iv) assure the nonemptyness of the solution set of (\ref{p2}). On the other hand, (iii), (v) and (vi) via Theorem \ref{t3.0} assure the nonemptyness of the solution set of (\ref{p1}).
  \end{proof}

\section{The case of reflexive Banach spaces}

Note that Condition (iv) in the hypotheses of Theorem \ref{t11}, Theorem \ref{t110} and Theorem \ref{t112}  is usually  hard to be verified.  However, it is well known that in a reflexive Banach space $X$, the closed ball with radius $r>0$, $\overline{B}_r:=\{x\in X:\|x\|\le r\},$ is weakly compact. Therefore, if we endow the reflexive Banach space $X$ with the weak topology, we can take $K_0=\overline{B}_r\cap K$, hence, condition (iv) in Theorem \ref{t11} becomes :
there exists $r>0\mbox{ and }y_0\in K,$ such that for all $x\in K\mbox{ satisfying }\|x\|>r$ one has that $F(x,y_0,x)\in-\inte C.$

Furthermore, in this setting condition (iv) in the hypothesis of Theorem \ref{t11} can be weakened by assuming that  there exists $r>0,$ such that for all $x\in K$ satisfying %
$\|x\|>r$, there exists some $y_0\in K,$ (which may depend by $x$), with $\|y_0\|<\|x\|$ and for which the  condition
$F(x,y_0,x)\in- C$ holds. More precisely, we have the following result.


\begin{theorem}
\label{t12} Let $X$ be a reflexive Banach space and let $Z$ be  a Hausdorff topological vector space.  Let $C\subseteq Z$ be a convex and pointed cone with nonempty interior, and let $K$ be a nonempty, convex and closed subset of $X$. Consider the mapping $F :K \times K\times K \longrightarrow{Z}$  satisfying

\begin{itemize}
\item[(i)] $\forall y\in K,$ one of the conditions (a), (b), (c) in Lemma \ref{primclosed} is satisfied, with respect to  the weak topology of $X$,

\item[(ii)] $\forall x\in K,$ the mapping $y\longrightarrow F(x,y,x)$ is C-convex on $K,$

\item[(iii)] $\forall x\in K,\, F(x,x,x)\in -C\setminus-\inte C,$

\item[(iv)] $\exists r>0$ such that, for all $x\in K$, $\|x\|>r$, there exists $y_0\in K$ with $\|y_0\|<\|x\|,$ such that
$F(x,y_0,x)\in-C.$
\end{itemize}

Then, there exists an element $x_0\in K,$ such that $F(x_0,y,x_0)\not\in -\inte{C}$ for all $y\in K.$
\end{theorem}
\begin{proof} Let $r>0$ such that (iv) holds, and let $r_1>r.$ Consider $K_0=K\cap \overline{B}_{r_1}.$ Obviously, $K_0$ is weakly compact, hence, according to Theorem \ref{t11} there exists $x_0\in K_0$ such that $F(x_0,y,x_0)\not\in-\inte C$ for all $y\in K_0.$ We show, that $F(x_0,y,x_0)\not\in-\inte C$ for all $y\in K.$
First we show, that there exists $z_0\in K_0,\,\|z_0\|<r_1$ such that $F(x_0,z_0,x_0)\in-C.$ Indeed, if $\|x_0\|<r_1$ then let $z_0=x_0$ and the conclusion follows from (iii). Assume now, that $\|x_0\|=r_1>r.$ Then, according to (iv),  there exists $z_0\in K,\,\|z_0\|<\|x_0\|=r_1$ such that $F(x_0,z_0,x_0)\in-C.$ 

Let $y\in K.$ Then, there exists $\l\in]0,1[$ such that $\l z_0 +(1-\l)y\in K_0,$ consequently $F(x_0,\l z_0 +(1-\l)y,x_0)\not\in-\inte C.$ From (ii) we have
$$\l F(x_0,z_0,x_0)+(1-\l)F(x_0,y,x_0)-F(x_0,\l z_0 +(1-\l)y,x_0)\in C$$
or, equivalently
$$(1-\l)F(x_0,y,x_0)-F(x_0,\l z_0 +(1-\l)y,x_0)\in C.$$
Assume that $F(x_0,y,x_0)\in-\inte C.$ Then, $$-F(x_0,\l z_0 +(1-\l)y,x_0)\in -(1-\l)F(x_0,y,x_0)+ C\subseteq \inte C,$$ in other words
$$F(x_0,\l z_0 +(1-\l)y,x_0)\in-\inte C,$$
contradiction. Hence,
$F(x_0,y,x_0)\not\in-\inte C$ for all $y\in K.$
\end{proof}

\begin{remark}\rm  In what follows we provide another coercivity condition (iv) which ensures the solution existence in a reflexive Banach space context. More precisely, we assume that there exists $r>0,$ such that, for all $x\in K$ satisfying $ \|x\|\le r,$ there exists $y_0\in K$ with $\|y_0\|<r,$  and $F(x,y_0,x)\in-C.$ Note that the diagonal condition (iii) is more general than the one assumed in Theorem \ref{t12}.
\end{remark}

The following result holds.
\begin{theorem}
\label{t13} Let $X$ be a reflexive Banach space and let $Z$ be  a Hausdorff topological vector space.  Let $C\subseteq Z$ be a convex and pointed cone with nonempty interior, and let $K$ be a nonempty, convex and closed subset of $X$. Consider the mapping $F :K \times K \times K\longrightarrow{Z}$  satisfying

\begin{itemize}
\item[(i)] $\forall y\in K,$ one of the conditions (a), (b), (c) in Lemma \ref{primclosed} is satisfied with respect to the weak topology of $X$,

\item[(ii)] $\forall x\in K,$ the mapping $y\longrightarrow F(x,y,x)$ is C-convex,

\item[(iii)] $\forall x\in K,\, F(x,x,x)\not\in-\inte C,$

\item[(iv)] $\exists r>0$ such that, for all $x\in K$, $\|x\|\le r$, there exists $y_0\in K$ with $\|y_0\|<r,$ and
$F(x,y_0,x)\in-C.$
\end{itemize}

Then, there exists an element $x_0\in K,$ such that $F(x_0,y,x_0)\not\in -\inte{C}$ for all $y\in K.$
\end{theorem}
\begin{proof}  Let $r>0$ such that (iv) holds, and consider $K_0=K\cap \overline{B}_{r}.$ Obviously, $K_0$ is weakly compact, hence, according to Theorem \ref{t11} there exists $x_0\in K_0$ such that $F(x_0,y,x_0)\not\in-\inte C$ for all $y\in K_0.$ We show, that $F(x_0,y,x_0)\not\in-\inte C$ for all $y\in K.$
According to (iv) there exists $z_0\in K$ with $\|z_0\|<r,$ such that $F(x_0,z_0,x_0)\in-C.$ On the other hand, $z_0\in K_0,$ hence  $F(x_0,z_0,x_0)\not\in-\inte C.$ Consequently $F(x_0,z_0,x_0)\in-C\setminus-\inte C.$
Let $y\in K\setminus K_0.$ Then,  there exists $\l\in]0,1[$ such that $\l z_0 +(1-\l)y\in  K_0,$ consequently $F(x_0,\l z_0 +(1-\l)y,x_0)\not\in-\inte C.$ From (ii) we have
$$\l F(x_0,z_0,x_0)+(1-\l)F(x_0,y,x_0)-F(x_0,\l z_0 +(1-\l)y,x_0)\in C$$
or, equivalently
$$(1-\l)F(x_0,y,x_0)-F(x_0,\l z_0 +(1-\l)y,x_0)\in C.$$
Assume that $F(x_0,y,x_0)\in-\inte C.$ Then, $$-F(x_0,\l z_0 +(1-\l)y,x_0)\in -(1-\l)F(x_0,y,x_0)+ C\subseteq \inte C,$$ or, in other words
$$F(x_0,\l z_0 +(1-\l)y,x_0)\in-\inte C,$$
contradiction. Hence, $F(x_0,y,x_0)\not\in-\inte C$ for all $y\in K.$
\end{proof}
In what follows, we reformulate Theorem \ref{t110}  for a reflexive Banach space setting. Note that we need to assume the closedness of $K$ in order to obtain the weak compactness of the intersection of a closed ball with $K.$ The condition (iv) becomes slightly weaker than in Theorem \ref{t13}, however we need a stronger diagonal condition (iii). Taking into account that for $r>0,$ $K_0=B_r\cap K=\{x\in K:\|x\|\le r\},$ $\mbox{core}_K K_0=\{x\in K:\|x\|<r\}$ and $K\setminus \mbox{core}_K K_0=\{x\in K:\|x\|=r\}$, the new condition (iv) becomes: there exists $r>0$ such that for all $x\in K$, $\|x\|= r$, there exists $y_0\in K$ with $\|y_0\|<r$ and
$F(x,y_0,x)\in-C.$

\begin{theorem}
\label{t134} Let $X$ be a reflexive Banach space and let $Z$ be  a Hausdorff topological vector space.  Let $C\subseteq Z$ be a convex and pointed cone with nonempty interior, and let $K$ be a nonempty, convex and closed subset of $X$. Consider the mapping $F :K \times K \times K\longrightarrow{Z}$  satisfying

\begin{itemize}
\item[(i)] $\forall y\in K,$ one of the conditions (a), (b), (c) in Lemma \ref{primclosed} is satisfied with respect to the weak topology of $X,$

\item[(ii)] $\forall x\in K,$ the mapping $y\longrightarrow F(x,y,x)$ is C-convex,

\item[(iii)] $\forall x\in K,\, F(x,x,x)\in -C\setminus -\inte C,$

\item[(iv)] $\exists r>0$ such that for all $x\in K$, $\|x\|= r$, there exists $y_0\in K$ with $\|y_0\|<r$ and
$F(x,y_0,x)\in-C.$
\end{itemize}

Then, there exists an element $x_0\in K,$ such that $F(x_0,y,x_0)\not\in -\inte{C}$ for all $y\in K.$
\end{theorem}
  \begin{remark}\rm One can easily obtain solution existence of (\ref{p2}) under the extra condition that $F$ is weakly C-pseudomonotone with respect to the third variable assumed in the hypotheses of Theorem \ref{t12}, Theorem \ref{t13} and Theorem \ref{t134}. Moreover, it is an easy exercise to reformulate Theorem \ref{t112} in the reflexive Banach space setting in order to obtain solution existence of (\ref{p1}) from the nonemptyness of the solution set of (\ref{p2}).
  \end{remark}

\section{On the perturbed weak vector equilibrium problems}

In this section we obtain solution existence of a perturbed weak vector equilibrium problem. Let $X$ and $Z$ be  Hausdorff topological vector spaces and  $K$ be a nonempty, convex and closed subset of $X$. We consider further  $C\subseteq Z$  a convex and pointed cone with nonempty interior.

Let $f:K\times K\To Z$ be a bifunction and assume that $f$ is diagonal null, that is $f(x,x)=0$ for all $x\in K.$ Consider the weak vector equilibrium problem, which consists in finding $x_0\in K$ such that
\begin{equation}\label{B}
f(x_0,y)\not\in -\inte C,\,\forall y\in K.
\end{equation}
Let $g:K\times K\To Z$ be another bifunction, We associate to (\ref{B}) the following perturbed vector equilibrium problem. Find $x_0\in K$ such that \begin{equation}\label{Per}
f(x_0,y)+g(x_0,y)\not\in-\inte C,\,\forall y\in K.
\end{equation}

As it was emphasized before (\ref{Per}) can be considered as a particular case of the primal problem (\ref{p1}) with the trifunction $F_1:K\times K\times K\to Z,\,F_1(x,y,z)=f(z,y)+g(x,y).$ Note that in this case the dual of (\ref{Per}) is the following problem.
Find $x_0\in K$ such that
\begin{equation}\label{g}
g(x_0,y)\not\in -\inte C,\,\forall y\in K.
\end{equation}

On the other hand, (\ref{Per}) can be considered as a particular instance of the dual problem (\ref{p2}) with the trifunction $F_2:K\times K\times K\to Z,\,F_2(x,y,z)=f(x,z)+g(x,y).$ In this case the primal problem is given by (\ref{g}).

 Hence, by using the results from the previous sections one can easily obtain solution existence for (\ref{Per}). For instance, it is an easy exercise that the $C-$convexity of the mappings $y\To f(x,y)$ and $y\To g(x,y)$ for every $x\in K$ assure the $C-$convexity of the mapping $y\To F_1(x,y,x)$ for every $x\in K$ and the $C-$convexity of the mapping $y\To F_2(x,y,y)$ for every $x\in K$, respectively. We will use condition (c) of Lemma \ref{primclosed}, since this assumption is new in the literature and, as it was shown in Example \ref{ex3.2}, it is also weaker than C-upper semicontinuity. An easy consequence of Theorem \ref{t11} is the following result.

\begin{theorem}
\label{t51} Let $X$ and $Z$ be  Hausdorff topological vector spaces,  let $C\subseteq Z$ be a convex and pointed cone with nonempty interior, and let $K$ be a nonempty, convex and closed subset of $X$. Consider the mappings $f,g :K \times K\longrightarrow{Z}$  satisfying
\begin{itemize}
\item[(i)] $\forall y\in K,$ it holds that for every  $x\in K$ and for every  net $(x_\a)\subseteq K,\, \lim x_\a=x$ there exists a net $z_\a\subseteq Z,\, \lim z_\a=z$ such that $f(x_\a,y)+g(x_\a,y)-z_\a\in -C$ and $f(x,y)+g(x,y)-z\in C,$

\item[(ii)] $\forall x\in K,$ the mappings $y\longrightarrow f(x,y)$ and $y\longrightarrow g(x,y)$ are C-convex,

\item[(iii)] $\forall x\in K,\, f(x,x)=0\mbox{ and }g(x,x)\not\in-\inte  C,$

\item[(iv)] There exists $K_0\subseteq X$ a nonempty and compact set and $y_0\in K,$ such that $f(x,y_0)+g(x,y_0)\in-\inte C,$ for all $x\in K\setminus K_0.$
\end{itemize}

Then, there exists an element $x_0\in K$ such that $f(x_0,y)+g(x_0,y)\not\in -\inte{C},$ for all $y\in K.$
\end{theorem}
\begin{proof} The conclusion follows by  Theorem \ref{t11} by taking $F_1(x,y,z)=f(z,y)+g(x,y)$ in its hypothesis.
\end{proof}
\begin{remark}\rm Note that condition (i) in Theorem \ref{t51} is satisfied if we assume separately for the bifunctions $f$ and $g$ the following: for all $y\in K,$ it holds that for every  $x\in K$ and for every  net $(x_\a)\subseteq K,\, \lim x_\a=x$ there exist the nets $z_\a^1,\,z_\a^2\subseteq Z,\, \lim z_\a^1=z^1,\,\lim z_\a^2=z^2$ such that $f(x_\a,y)-z_\a^1\in -C,\,g(x_\a,y)-z_\a^2\in -C$ and $f(x,y)-z^1\in C,\,g(x,y)-z^2\in C.$
\end{remark}

\begin{remark}\rm Solution existence of (\ref{Per})  also follows via Theorem \ref{t110}, if we replace the conditions (iii) and (iv) in the hypothesis of Theorem \ref{t51} by the following.

(iii') $\forall x\in K,\, f(x,x)=0\mbox{ and }g(x,x)\in-C\setminus-\inte  C,$

(iv') There exists a nonempty compact convex subset $K_0$ of $K$ with the property that for every $x\in K_0\setminus\mbox{core}_K K_0,$ there exists an $y_0\in \mbox{core}_K K_0$ such that $f(x,y_0)+g(x,y_0)\in -C.$

Moreover, in this case we can drop the assumption that $K$ is closed.
\end{remark}

Next we obtain solution existence of the perturbed problem (\ref{Per}) via duality. Note that in this case the conditions can be  assumed not for all $x\in K,$ but relative to the solution of (\ref{g}).

\begin{theorem}\label{t5.2}   Let $X$ and $Z$ be  Hausdorff  topological vector spaces,  let $C\subseteq Z$ be a convex and pointed cone with nonempty interior and let $K$ be a nonempty convex subset of $X$. Consider the mappings $f,g :K \times K \longrightarrow{Z}$. Let $x_0$ be a  solution of the problem (\ref{g}), i.e. $g(x_0,y)\not\in-\inte C$ for all $y\in K.$ Assume that the following statements hold.
\begin{itemize}
\item[(i)] For all $y\in K$ and $t\in]0,1[$ one has that $g(x_0,(1-t)x_0+ty)-f((1-t)x_0+ty,y)-g(x_0,y)\in-\inte C$,
 \item[(ii)]For every $y\in K$ the following implication holds. If $f((1-t)x_0+ty,y)+g(x_0,y)\not\in-\inte C$ for all $t\in]0,1],$ then $f(x_0,y)+g(x_0,y)\not\in-\inte C.$
\item[(iii)]  For all $x\in K,\,f(x,x)=0.$
     \end{itemize}
 Then, $x_0$ is a solution of (\ref{Per}), that is, $f(x_0,y)+g(x_0,y)\not\in-\inte C$ for all $y\in K.$
\end{theorem}
\begin{proof}
Let $y\in K.$ Since $x_0$ is a solution of (\ref{g}) one has, that $g(x_0,(1-t)x_0+ty)\not\in-\inte C$ for all $t\in[0,1]$. Hence, by using the fact that $C+\inte C=\inte C$, from (i) we have that
$f((1-t)x_0+ty,y)+g(x_0,y)\not\in-C,$ for all $t\in]0,1[$ and $y\in K.$ On the other hand, $f(y,y)+g(x_0,y)=g(x_0,y)\not\in-\inte C$, hence
$$f((1-t)x_0+ty,y)+g(x_0,y)\not\in-\inte C,\mbox{ for all }t\in]0,1].$$
From (ii) we obtain that  $f(x_0,y)+g(x_0,y)\not\in-\inte C.$ Since $y\in K$ was arbitrary chosen the conclusion follows.
\end{proof}

In what follows we obtain solution existence of (\ref{Per}) by assuming different conditions for the bifunctions $f$ and $g.$  We need the following notion.

\begin{definition}
A bifunction $f: K\times K\To Z$ is said to be C-essentially quasimonotone relative to the second variable, iff
for all $y_1,y_2,...,y_n\in K$ and all $\l_1,\l_2,...,\l_n\ge 0$ with $\sum_{i=1}^n\l_i=1$  one has
$$\sum_{i=1}^n \l_i f\left(\sum_{i=1}^n \l_i y_i,y_i\right)\not\in-\inte C.$$
\end{definition}
\begin{lemma}\label{difrence} Let $X$ and $Z$ be  Hausdorff  topological vector spaces,  let $C\subseteq Z$ be a convex and pointed cone with nonempty interior and let $K$ be a nonempty, convex subset of $X$. Consider the mapping $F_1:K \times K\times K \longrightarrow{Z},\, F_1(x,y,z)=f(z,y)+g(x,y)$ and assume that the bifunctions $f,g:K\times K\To Z$ satisfy

\begin{itemize}
\item[(i)] $f$ is C-essentially quasimonotone relative to the second variable,

\item[(ii)] $y\To g(x,y)$ is C-convex for all $x\in K$ and $g(x,x)\in C$ for all $x\in K.$
\end{itemize}
Then, the map $G:K\rightrightarrows K,\,G(y)=\{x\in K: F_1(x,y,x)\not\in-\inte C\}$ is a KKM application.
\end{lemma}
\begin{proof}
We show at first that  for all $y_1,y_2,...,y_n\in K$ and $\l_1,\l_2,...,\l_n\ge 0$ with $\sum_{i=1}^n \l_i=1,\,n\ge 1$ one has
$$\sum_{i=1}^n \l_i F_1\left(\sum_{i=1}^n \l_i y_i,y_i,\sum_{i=1}^n \l_i y_i\right)\not\in-\inte C.$$
Assume the contrary, that is, there exist $y_1,y_2,...,y_n\in K$ and  there exist $\l_1,\l_2,\ldots,\l_n\ge 0,$ with $\sum_{i=1}^n\l_i=1$ such that
$$\sum_{i=1}^n \l_i F_1\left(\sum_{i=1}^n\l_i y_i,y_i,\sum_{i=1}^n\l_i y_i\right)\in-\inte C.$$
This assumption is equivalent to
$$\sum_{i=1}^n \l_i \left(f\left(\sum_{i=1}^n\l_i y_i,y_i\right)+g\left(\sum_{i=1}^n\l_i y_i,y_i\right)\right)\in-\inte C.$$

From assumption (i)  we have that $\sum_{i=1}^n \l_i f\left(\sum_{i=1}^n\l_i y_i,y_i\right)\not\in-\inte C.$ But then, since $\inte C+C=\inte C$, we have
$$\sum_{i=1}^n \l_i g\left(\sum_{i=1}^n\l_i y_i,y_i\right)\not\in C.$$
Now using the fact that $y\To g\left(\sum_{i=1}^n\l_i y_i,y\right)$ is C-convex  and $g(x,x)\in C$ for all $x\in K$ we obtain
$$\sum_{i=1}^n \l_i g\left(\sum_{i=1}^n\l_i y_i,y_i\right)-g\left(\sum_{i=1}^n\l_i y_i,\sum_{i=1}^n\l_i y_i\right)\in C,$$
contradiction.

Assume that $G:K\rightrightarrows K,\,G(y)=\{x\in K: F(x,y,x)\not\in-\inte C\}$ is not a KKM application. Then there exist $y_1,y_2,...,y_n\in K$ and $y\in \co\{y_1,y_2,...,y_n\}$ such that $y\not\in\cup_{i=1}^n G(y_i).$ 

In other words,
there exist $\l_1,\l_2,...,\l_n\ge 0$ with $\sum_{i=1}^n \l_i=1$ such that $y=\sum_{i=1}^n \l_i y_i\not\in G(y_i)$ for all $i\in\{1,2,...,n\},$ that is
$$F\left(\sum_{i=1}^n \l_i y_i,y_i,\sum_{i=1}^n \l_i y_i\right)\in-\inte C,\, \forall i\in\{1,2,...,n\}.$$
But then, since $-\inte C$ is convex one has
$$\sum_{i=1}^n \l_i F\left(\sum_{i=1}^n \l_i y_i,y_i,\sum_{i=1}^n \l_i y_i\right)\in-\inte C,$$
which contradicts the fact that $$\sum_{i=1}^n \l_i F_1\left(\sum_{i=1}^n \l_i y_i,y_i,\sum_{i=1}^n \l_i y_i\right)\not\in-\inte C.$$
\end{proof}

An easy consequence is the following theorem.

\begin{theorem}\label{t5.3} Let $X$ and $Z$ be  Hausdorff  topological vector spaces,  let $C\subseteq Z$ be a convex and pointed cone with nonempty interior and let $K$ be a nonempty convex  subset of $X$. Assume that the bifunctions $f,g:K\times K\To Z$ satisfy

\begin{itemize}
\item[(i)] There exists a nonempty compact convex subset $K_0$ of $K$ with the property that for every $x\in K_0\setminus\mbox{core}_K K_0,$ there exists an $y_0\in \mbox{core}_K K_0$ such that $f(x,y_0)+g(x,y_0)\in -C.$

 \item[(ii)] $\forall y\in K_0,$ it holds that for every  $x\in K_0$ and for every  net $(x_\a)\subseteq K_0,\, \lim x_\a=x$ there exists a net $z_\a\subseteq Z,\, \lim z_\a=z$ such that $f(x_\a,y)+g(x_\a,y)-z_\a\in -C$ and $f(x,y)+g(x,y)-z\in C,$

\item[(iii)] $f$ is C-essentially quasimonotone relative to the second variable on $K_0$, that is for all $y_1,y_2,...,y_n\in K_0$ and all $\l_1,\l_2,...,\l_n\ge 0$ with $\sum_{i=1}^n\l_i=1$  one has
$$\sum_{i=1}^n \l_i f\left(\sum_{i=1}^n \l_i y_i,y_i\right)\not\in-\inte C.$$

\item[(iv)] $y\To g(x,y)$ and $y\To f(x,y)$ are C-convex on $K$ for all $x\in K_0,$

\item[(v)] $\forall x\in K_0,\, f(x,x)\in-C\setminus-\inte C\mbox{ and }g(x,x)=0$.

\end{itemize}
Then, there exists $x_0\in K,$ such that $f(x_0,y)+g(x_0,y)\not\in-\inte C$ for all $y\in K.$
\end{theorem}
\begin{proof} Consider the mapping $F:K_0 \times K_0\times K_0 \longrightarrow{Z},\, F(x,y,z)=f(z,y)+g(x,y)$. Lemma \ref{difrence} assures that $$G:K_0\toto K_0,\,\, G(y)=\{x\in K_0:F(x,y,x)\not\in-\inte C\}$$ is a KKM mapping. On the other hand, (i) assures that $G(y)$ is closed for every $y\in K_0.$ Since $K_0$ is compact we have that $G(y)$ is compact for every $y\in K_0,$ hence according to Lemma \ref{fan}, $\cap_{y\in K_0}G(y)\neq\emptyset.$ In other words, there exists $x_0\in K_0$ such that $f(x_0,y)+g(x_0,y)\not\in-\inte C$ for all $y\in K_0.$

We show that the latter relation holds for every $y\in K.$ First we show, that there exists $z_0\in \mbox{core}_K K_0$ such that $f(x_0,z_0)+g(x_0,z_0)\in -C.$ Indeed, if $x_0\in \mbox{core}_K K_0$ then let $z_0=x_0$ and the conclusion follows from (v). Assume now, that $x_0\in K_0\setminus\mbox{core}_K K_0. $ Then, according to (i),  there exists $z_0\in \mbox{core}_K K_0$ such that $f(x_0,z_0)+g(x_0,z_0)\in -C.$

Let $y\in K.$ Then, since $z_0\in \mbox{core}_K K_0$, there exists $\l\in[0,1]$ such that $\l z_0 +(1-\l)y\in K_0,$ consequently $f(x_0,\l z_0 +(1-\l)y)+g(x_0,\l z_0 +(1-\l)y)\not\in-\inte C.$ From (iv) we have
$$\l(f(x_0, z_0)+g(x_0, z_0))+(1-\l)(f(x_0,y)+g(x_0,y))-$$
$$(f(x_0,\l z_0 +(1-\l)y)+g(x_0,\l z_0 +(1-\l)y))\in C$$
or, equivalently
$$(1-\l)(f(x_0,y)+g(x_0,y))-(f(x_0,\l z_0 +(1-\l)y)+g(x_0,\l z_0 +(1-\l)y))\in$$
$$ C-\l (f(x_0, z_0)+g(x_0, z_0)\subseteq C.$$
Assume that $f(x_0,y)+g(x_0,y)\in-\inte C.$ Then, $$-(f(x_0,\l z_0 +(1-\l)y)+g(x_0,\l z_0 +(1-\l)y))\in$$
 $$-(1-\l)(f(x_0,y)+g(x_0,y))+ C\subseteq \inte C,$$ in other words
$$f(x_0,\l z_0 +(1-\l)y)+g(x_0,\l z_0 +(1-\l)y)\in-\inte C,$$
contradiction. Hence,
$f(x_0,y)+g(x_0,y)\not\in-\inte C,$ for all $y\in K.$
\end{proof}
\begin{remark}\rm If $K$ is also compact, then one can take $K_0=K$, thus, one can drop the assumption (i)  and the assumption that the map $y\To f(x,y)$ is C-convex for every $x\in K$ in the hypothesis of Theorem \ref{t5.3}. Moreover, the assumptions imposed on the bifunctions $f$ and $g$ can be permuted, which might become useful in order to chose the right perturbation bifunction,  when we perturb a concrete problem.
\end{remark}

\end{document}